\documentclass[12pt]{amsart}
\usepackage{amsmath,amssymb,amsthm,amscd,amsfonts,geometry,color}
\renewcommand{\labelenumi}{\rm(\theenumi)}

\theoremstyle{plain}
\newtheorem{thm}{Theorem}[section]
\newtheorem{prop}[thm]{Proposition}

\newtheorem{lem}[thm]{Lemma}

\newtheorem{main}{Main Theorem}

\theoremstyle{remark}
\newtheorem{remark}{Remark}

\newtheorem{example}{Example}

\theoremstyle{definition}

\newcommand{\R}{\mathbb{R}}                       
\newcommand{\N}{\mathbb{N}}                       

\newcommand{\Q}{\mathbf{Q}}                       
\newcommand{\s}{\mathbf{s}}                       
\newcommand{\cs}{\mathbf{c}}                      

\newcommand{\cl}{\operatorname{cl}}               
\newcommand{\intr}{\operatorname{int}}            

\newcommand{\diam}{\operatorname{diam}}           

\newcommand{\cpt}{\operatorname{Comp}}             

\newcommand{\dom}{\operatorname{dom}}            

\begin{document}

\title[Spaces of metrics on locally compact metrizable spaces]{The topological type of spaces consisting of certain metrics on locally compact metrizable spaces with the compact-open topology}
\author{Katsuhisa Koshino}
\address[Katsuhisa Koshino]{Faculty of Engineering, Kanagawa University, Yokohama, 221-8686, Japan}
\email{ft160229no@kanagawa-u.ac.jp}
\maketitle

\begin{abstract}
For a separable locally compact but not compact metrizable space $X$, let $\alpha X = X \cup \{x_\infty\}$ be the one-point compactification with the point at infinity $x_\infty$.
We denote by $EM(X)$ the space consisting of admissible metrics on $X$,
 which can be extended to an admissible metric on $\alpha X$,
 endowed with the compact-open topology.
Let $\cs_0 \subset (0,1)^\N$ be the space of sequences converging to $0$.
In this paper, we shall show that if $X$ is separable, locally connected and locally compact but not compact,
 and there exists a sequence $\{C_i\}$ of connected sets in $X$ such that for all positive integers $i, j \in \N$ with $|i - j| \leq 1$, $C_i \cap C_j \neq \emptyset$,
 and for each compact set $K \subset X$, there is a positive integer $i(K) \in \N$ such that for any $i \geq i(K)$, $C_i \subset X \setminus K$,
 then $EM(X)$ is homeomorphic to $\cs_0$.
\end{abstract}

\section{Introduction}

Throughout this paper, spaces are separable metrizable, maps are continuous,
 but functions are not necessarily continuous.
Let $\R$ be the space of real numbers with the usual metric,
 $\N$ be the set of positive integers,
 and $X$ be a locally compact space.
We denote by $C(X^2)$ the space of continuous real-valued functions on $X^2$ equipped with the compact-open topology.
In the case that $X$ is locally connected,
 $C(X^2)$ is a Fr\'{e}chet space.
Let $PM(X)$, $M(X)$ and $AM(X)$ be the spaces of continuous pseudometrics, continuous metrics and admissible metrics on $X$ with the relative topology of $C(X^2)$, respectively.
As is easily observed,
 $PM(X)$ is a convex non-negative cone,
 and $M(X)$ and $AM(X)$ are convex positive cones in $C(X^2)$.
Recall that $M(X) = AM(X)$ when $X$ is compact.
However, they are not necessarily coincident in general.

\begin{example}
On the half open interval $[0,1)$ topologized by the usual metric, define a continuous metric $d$ as follows:
 $$d(x,y) = \min\{|x - y|, 1 + x - y, 1 - x + y\}$$
 for all $x, y \in [0,1)$.
Then $d$ is not admissible.
\end{example}

If $X$ is not compact,
 then it has the one-point compactification $\alpha X = X \cup \{x_\infty\}$ with the point at infinity $x_\infty$ topologized by the following collection
 $$\{U, (X \setminus K) \cup \{x_\infty\} \mid U \text{ is open in } X, K \text{ is compact in } X\}.$$
In the paper, we shall investigate the topological type of the following subspace of $AM(X)$:
 $$EM(X) = \{d \in AM(X) \mid d \text{ can be extended to an admissible metric on } \alpha X\}.$$
Topologies of function spaces have been studied in the theory of infinite-dimensional topology.
We denote the Hilbert cube by $\Q = [0,1]^\N$ and the pseudo interior by $\s = (0,1)^\N$.
Let
 $$\cs_0 = \bigg\{(x(n))_{n \in \N} \in \s \ \bigg| \ \lim_{n \to \infty} x(n) = 0\bigg\},$$
 that is homeomorphic to several function spaces, refer to \cite{Mil3,Ca,YZ,YSK,Kos14}.
We will establish the following:

\begin{main}
Let $X$ be a locally connected but not compact space.
If there is a sequence $\{C_i\}$ consisting of connected sets in $X$ such that for all $i, j \in \N$ with $|i - j| \leq 1$, $C_i \cap C_j \neq \emptyset$,
 and for each compact subset $K$ of $X$, there exists $i(K) \in \N$ such that for every $i \geq i(K)$, $C_i \subset X \setminus K$,
 then $EM(X)$ is homeomorphic to $\cs_0$.
\end{main}

\section{Preliminaries}

Given spaces $A \subset Y$, denote the interior of $A$ by $\intr{A}$ and the closure of $A$ by $\cl{A}$.
For functions $f : Z \to Y$ and $g : Z \to Y$, and for an open cover $\mathcal{U}$ of $Y$, $f$ is said to be \textit{$\mathcal{U}$-close} to $g$ if for every $z \in Z$, there is $U \in \mathcal{U}$ containing $f(z)$ and $g(z)$.
A closed set $A \subset Y$ is a \textit{$Z$-set} if for each open cover $\mathcal{U}$ of $Y$, there exists a map $f : Y \to Y$ such that $f$ is $\mathcal{U}$-close to the identity map and $f(Y) \cap A = \emptyset$.
A \textit{$Z_\sigma$-set} is a countable union of $Z$-sets.
A \textit{$Z$-embedding} is an embedding whose image is a $Z$-set.
Given a class $\mathfrak{C}$ of spaces, we call $Y$ to be \textit{strongly $\mathfrak{C}$-universal} if the following condition holds.
\begin{itemize}
 \item Let $A$ be a space in $\mathfrak{C}$ and $f : A \to Y$ be a map.
 Assume that $B$ is a closed set in $A$ and the restriction $f|_B$ is a $Z$-embedding.
 Then for each open cover $\mathcal{U}$ of $Y$, there exists a $Z$-embedding $g : A \to Y$ such that $g$ is $\mathcal{U}$-close to $f$ and $g|_B = f|_B$.
\end{itemize}
Let $\mathfrak{C}_\sigma$ denote the class of spaces that are countable unions of closed subspaces belonging to $\mathfrak{C}$.
For spaces $Y \subset M$, $Y$ is \textit{homotopy dense} in $M$ provided that $M$ has a homotopy $h : M \times [0,1] \to M$ such that $h(M \times (0,1]) \subset Y$ and $h(y,0) = y$ for any $y \in M$.
A space $Y$ is said to be a \textit{$\mathfrak{C}$-absorbing set} in $M$ if it satisfies the following conditions.
\begin{enumerate}
 \item $Y \in \mathfrak{C}_\sigma$ and is homotopy dense in $M$.
 \item $Y$ is strongly $\mathfrak{C}$-universal.
 \item $Y$ is contained in some $Z_\sigma$-set in $M$.
\end{enumerate}
The symbol $\mathfrak{M}_2$ stands for the class of \textit{absolute $F_{\sigma\delta}$-spaces},
 that is, $Y \in \mathfrak{M}_2$ if $Y$ is an $F_{\sigma\delta}$-set in any space $M$ which contains $Y$.
It is known that $\cs_0$ is an $\mathfrak{M}_2$-absorbing set in $\s$.
Theorem~3.1 of \cite{BeMo} shows the topological uniqueness of absorbing sets in $\s$.

\begin{thm}\label{abs.}
For subspaces $Y, Z \subset \s$, if the both $Y$ and $Z$ are $\mathfrak{M}_2$-absorbing sets in $\s$,
 then $Y$ and $Z$ are homeomorphic.
\end{thm}

In the paper \cite{Kos16}, the author investigates the topological types of $PM(X)$ and $M(X)$,
 which are endowed with the uniform convergence topology.
By the same argument as it, we can establish the following theorem.

\begin{thm}\label{pm}
Let $X$ be a locally connected space.
If $X$ is not discrete,
 then both $PM(X)$ and $M(X)$ are homeomorphic to $\s$.
\end{thm}

A sequence $\{C_i\}$ of subsets in a space $Y$ is a \textit{simple chain} if for any $i, j \in \N$ with $|i - j| \leq 1$, $C_i \cap C_j \neq \emptyset$.
It is said that $Y$ is \textit{chain-connected to infinity} if there exists a simple chain $\{C_i\}$ of connected subsets in $Y$ such that for every compact set $K \subset Y$, there is $i(K) \in \N$ such that for any $i \geq i(K)$, $C_i \subset Y \setminus K$.

For a metric space $Y = (Y,d_Y)$, a subset $A \subset Y$ and a positive number $\delta > 0$, put
 $$B_{d_Y}(A,\delta) = \{y \in Y \mid d_Y(y,A) < \delta\} \text{ and } \overline{B}_{d_Y}(A,\delta) = \{y \in Y \mid d_Y(y,A) \leq \delta\}.$$
When $A = \{a\}$,
 we write $B_{d_Y}(a,\delta) = B_{d_Y}(\{a\},\delta)$ and $\overline{B}_{d_Y}(a,\delta) = \overline{B}_{d_Y}(\{a\},\delta)$ for simplicity.
Moreover, $\diam_{d_Y}(A)$ denotes the diameter of $A$.
For metric spaces $Y_i = (Y_i,d_{Y_i})$, $i = 1, \cdots, n$, we will use an admissible metric $d_{\prod_{i = 1}^n Y_i}$ on $\prod_{i = 1}^n Y_i$ defined by
 $$d_{\prod_{i = 1}^n Y_i}((y_1,\cdots,y_n),(z_1,\cdots,z_n)) = \max_{1 \leq i \leq n} d_{Y_i}(y_i,z_i)$$
 for any $(y_1,\cdots,y_n), (z_1,\cdots,z_n) \in \prod_{i = 1}^n Y_i$.
We denote by $\cpt(Y)$ the hyperspace consisting of non-empty compact sets in $Y$ endowed with the Vietoris topology.
Note that the topology of $\cpt(Y)$ is induced by the Hausdorff metric $(d_Y)_H$,
 that is defined as follows:
 $$(d_Y)_H(A,B) = \inf\{r > 0 \mid A \subset B_{d_Y}(B,r), B \subset B_{d_Y}(A,r)\}$$
 for all $A, B \in \cpt(Y)$, see \cite[Proposition~5.12.4]{Sakaik11}.
Under our assumption, we have the following proposition, refer to \cite[Corollary~1.11.4 and Proposition~1.11.13]{Mil3}:

\begin{prop}\label{hyp.}
If $Y$ is a connected, locally connected and locally compact space,
 then so is $\cpt(Y)$.
\end{prop}

\section{The Borel complexity of $EM(X)$ in $PM(X)$}

In this section, it will be shown that $EM(X) \in \mathfrak{M}_2$.
Since $X$ is a separable locally compact metrizable space,
 we can write $X = \bigcup_{n \in \N} X_n$,
 where each $X_n$ is compact and $X_n \subset \intr{X_{n + 1}}$.
For positive integers $n, m \in \N$, set
 $$A(n,m) = \{d \in PM(X) \mid d(X_n,X \setminus X_{n + 1}) \geq 1/m\},$$
 which is closed in $PM(X)$.
We shall prove the following:

\begin{prop}\label{Fsigmadelta}
The subset $AM(X)$ is an $F_{\sigma\delta}$-set in $PM(X)$.
\end{prop}

\begin{proof}
We can write
 $$AM(X) = M(X) \bigcap \Bigg(\bigcap_{n \in \N} \bigcup_{m \in \N} A(n,m)\Bigg).$$
Remark that $M(X)$ is $G_\delta$ in $PM(X)$, refer to the proof of \cite[Lemma~5.1]{Ish},
 and therefore $AM(X)$ is $F_{\sigma\delta}$ in $PM(X)$.
\end{proof}

Since $X$ is separable locally compact,
 $C(X^2)$ is completely metrizable.
Hence the both closed subset $PM(X)$ and $G_\delta$ subset $M(X)$ of $C(X^2)$ are Baire spaces.
Furthermore, the following holds:

\begin{prop}
The following are equivalent:
\begin{enumerate}
 \item $X$ is compact;
 \item $AM(X)$ is a Baire space.
\end{enumerate}
\end{prop}

\begin{proof}
If $X$ is compact,
 $AM(X)$ coincides with the Baire space $M(X)$.
We will prove the implication (2) $\Rightarrow$ (1).
Suppose that $X$ is not compact.
Observe that $AM(X) \subset \bigcup_{m \in \N} A(1,m)$.
To show that every $A(1,m) \cap AM(X)$ is nowhere dense, take any admissible metric $d \in A(1,m) \cap AM(X)$ and any neighborhood $U$ of $d$.
We can choose $k \geq m$ so that if for any $(x,y) \in X_k^2$, $d(x,y) = \rho(x,y)$,
 then $\rho \in U$.
Since $X$ is not compact,
 we may assume that $X_k \neq \emptyset$.
Moreover, $X \setminus X_{k + 1}$ is not empty.
Fixing points $w \in X_k$ and $z \in X \setminus X_{k + 1}$, define an admissible metric $\rho : (X_k \cup \{z\})^2 \to [0,\infty)$ on $X_k \cup \{z\}$ as follows:
 $$\rho(x,y) = \left\{
 \begin{array}{ll}
 d(x,y) & \text{if } (x,y) \in X_k^2,\\
 d(x,w) + 1/(m + 1) & \text{if } x \in X_k \text{ and } y = z,\\
 d(y,w) + 1/(m + 1) & \text{if } x = z \text{ and } y \in X_k,\\
 0 & \text{if } x = y = z.
 \end{array}
 \right.$$
Due to Hausdorff's metric extension theorem \cite{Hau}, $\rho$ can be extended to an admissible metric $\tilde{\rho} \in AM(X)$.
Note that for each $(x,y) \in X_k^2$, $\tilde{\rho}(x,y) = \rho(x,y) = d(x,y)$,
 so $\tilde{\rho} \in U$.
Moreover, we have
 $$\tilde{\rho}(X_k,X \setminus X_{k + 1}) \leq \tilde{\rho}(w,z) = \rho(w,z) = 1/(m + 1) < 1/m,$$
 which implies that $\tilde{\rho} \notin A(1,m) \cap AM(X)$.
Therefore $A(1,m) \cap AM(X)$ is nowhere dense,
 and hence $AM(X)$ is not a Baire space.
This is a contradiction.
Consequently, (2) $\Rightarrow$ (1) holds.
\end{proof}

When $X$ is not compact,
 the family $\{\alpha X \setminus X_n \mid n \in \N\}$ is an open neighborhood basis of the point $x_\infty$ in $\alpha X$ and $x_\infty$ is not isolated.
The space $EM(X)$ can be represented as follows:

\begin{lem}\label{diam.}
If $X$ is not compact,
 then
 $$EM(X) = \bigg\{d \in AM(X) \bigg| \lim_{n \to \infty} \diam_d(X \setminus X_n) = 0\bigg\}.$$
\end{lem}

\begin{proof}
It is easy to show that
 $$EM(X) \subset \bigg\{d \in AM(X) \bigg| \lim_{n \to \infty} \diam_d(X \setminus X_n) = 0\bigg\}.$$
Conversely, we will verify that the left hand side contains the right one.
Fix any $d \in \{d \in AM(X) \mid \lim_{n \to \infty} \diam_d(X \setminus X_n) = 0\}$.
Taking $x_n \in X \setminus X_n$, we can obtain a Cauchy sequence $\{d(x,x_n)\} \subset \R$ for each $x \in X$ because $\lim_{n \to \infty} \diam_d(X \setminus X_n) = 0$.
Let $\overline{d} : (\alpha X)^2 \to [0,\infty)$ be a function defined by
 $$\overline{d}(x,y) = \left\{
 \begin{array}{ll}
 d(x,y) &\text{if } x, y \in X,\\
 \lim_{n \to \infty} d(x,x_n) &\text{if } x \in X \text{ and } y = x_\infty,\\
 \lim_{n \to \infty} d(y,x_n) &\text{if } x = x_\infty \text{ and } y \in X,\\
 0 &\text{if } x = y = x_\infty.
 \end{array}
 \right.$$
Observe that for any $x \in X$, $\overline{d}(x,x_\infty) > 0$.
Indeed, $x \in X_m$ for some $m \in \N$.
Since $d \in AM(X)$,
 for every $n \geq m + 1$,
 $$d(x,x_n) \geq d(X_m,X \setminus X_{m + 1}) > 0,$$
 and hence
 $$\overline{d}(x,x_\infty) = \lim_{n \to \infty} d(x,x_n) \geq d(X_m,X \setminus X_{m + 1}) > 0.$$
As is easily observed,
 $\overline{d}$ is a metric on $\alpha X$.
We show that $\overline{d} \in M(\alpha X)$,
 that is, $\overline{d}$ is continuous.
Let any $(x,y) \in (\alpha X)^2$.
When $(x,y) \in X^2$,
 the continuity of $\overline{d}$ at $(x,y)$ follows from the one of $d$.
When $x \in X$ and $y = x_\infty$,
 for each $\epsilon > 0$, there exist a neighborhood $U \subset X$ of $x$ and a positive integer $m \in \N$ such that for any $z \in U$, $d(x,z) \leq \epsilon/2$,
 and $\diam_d(X \setminus X_m) \leq \epsilon/4$.
Verify that $\diam_{\overline{d}}(\alpha X \setminus X_m) = \diam_d(X \setminus X_m)$.
For every point $(z,w) \in U \times (\alpha X \setminus X_m)$, that is a neighborhood of $(x,y)$ in $(\alpha X)^2$,
 we have that for each $n \geq m$,
\begin{align*}
 |\overline{d}(z,w) - \overline{d}(x,y)| &\leq |\overline{d}(z,w) - \overline{d}(z,x_n)| + |\overline{d}(z,x_n) - \overline{d}(x,x_n)| + |\overline{d}(x,x_n) - \overline{d}(x,y)|\\
 &\leq \overline{d}(w,x_n) + \overline{d}(x,z) + \overline{d}(y,x_n)\\
 &\leq \diam_d(X \setminus X_m) + d(x,z) + \diam_{\overline{d}}(\alpha X \setminus X_m)\\
 &= 2\diam_d(X \setminus X_m) + d(x,z) \leq \epsilon/2 + \epsilon/2 = \epsilon.
\end{align*}
Hence $\overline{d}$ is continuous at $(x,y)$.
Similarly, the continuity of $\overline{d}$ at $(x,y)$, where $x = x_\infty$ and $y \in X$,
 is valid.
When $x = x_\infty$ and $y = x_\infty$,
 for each $\epsilon > 0$, there is $m \in \N$ such that $\diam_d(X \setminus X_m) \leq \epsilon$.
Then $(\alpha X \setminus X_m)^2$ is a neighborhood of $(x,y)$ and for each $(z,w) \in (\alpha X \setminus X_m)^2$,
 $$|\overline{d}(z,w) - \overline{d}(x,y)| = \overline{d}(z,w) \leq \diam_{\overline{d}}(\alpha X \setminus X_m) = \diam_d(X \setminus X_m) \leq \epsilon,$$
 which implies the continuity of $\overline{d}$ at $(x,y)$.
Thus $\overline{d}$ is continuous.
It follows from the compactness of $\alpha X$ that $\overline{d} \in AM(\alpha X)$.
As a consequense, $d$ is extended to the admissible metric $\overline{d}$,
 so $d \in EM(X)$.
The proof is completed.
\end{proof}

It is known that a space $Y \in \mathfrak{M}_2$ if and only if there exists an embedding from $Y$ into a completely metrizable space as an $F_{\sigma\delta}$-set,
 refer to \cite[Theorem~9.6]{AN}.

\begin{prop}\label{M2}
Suppose that $X$ is not a compact space.
Then the subset $EM(X)$ is $F_{\sigma\delta}$ in $PM(X)$,
 and hence it is in $\mathfrak{M}_2$.
\end{prop}

\begin{proof}
As is easily observed,
 $$\bigg\{d \in AM(X) \bigg| \lim_{n \to \infty} \diam_d(X \setminus X_n) = 0\bigg\} = \bigcap_{m \in \N} \bigcup_{n \in \N} \{d \in AM(X) \mid \diam_d(X \setminus X_n) \leq 1/m\},$$
 and for any $m, n \in \N$, $\{d \in AM(X) \mid \diam_d(X \setminus X_n) \leq 1/m\}$ is closed.
According to Lemma~\ref{diam.}, $EM(X)$ is an $F_{\sigma\delta}$-set in $AM(X)$.
Combining it with Proposition~\ref{Fsigmadelta}, we have that $EM(X)$ is $F_{\sigma\delta}$ in $PM(X)$.
\end{proof}

\section{The homotopy density of $EM(X)$ in $PM(X)$}

In this section, we shall prove that $EM(X)$ is homotopy dense in $PM(X)$.

\begin{prop}\label{dense}
Suppose that $X$ is not a compact space.
Then $EM(X)$ is dense in $PM(X)$.
\end{prop}

\begin{proof}
According to the same argument as Proposition~1 of \cite{Kos16}, $AM(X)$ is dense in $PM(X)$.
It remains to show that $EM(X)$ is dense in $AM(X)$.
For each $d \in AM(X)$ and each neighborhood $U$ of $d$, we shall construct a metric $\rho \in EM(X)$ such that $\rho \in U$.
There is a compact subset $K \subset X$ such that if $d|_{K^2} = \rho|_{K^2}$,
 then $\rho \in U$.
Define an admissible metric $\rho|_{(K \cup \{x_\infty\})^2}$ on $K \cup \{x_\infty\}$ as follows:
 $$\rho(x,y) = \left\{
 \begin{array}{ll}
 d(x,y) &\text{if } x, y \in K,\\
 d(x,x_0) + 1 &\text{if } x \in K \text{ and } y = x_\infty,\\
 d(y,x_0) + 1 &\text{if } x = x_\infty \text{ and } y \in K,\\
 0 &\text{if } x = y = x_\infty.
 \end{array}
 \right.$$
Due to Hausdorff's metric extension theorem \cite{Hau}, the above metric can be extended to an admissible metric $\rho$ on $\alpha X$.
The restriction $\rho|_{X^2} \in EM(X)$ is the desired admissible metric such that $\rho|_{X^2} \in U$.
The proof is complete.
\end{proof}

Applying Lemma~\ref{diam.}, we will verify the convexity of $EM(X)$.

\begin{prop}\label{cvx.}
If $X$ is not a compact space,
 then $EM(X)$ is a convex subset of $C(X^2)$.
Moreover, if $X$ is locally connected,
 then it is an AR.
\end{prop}

\begin{proof}
To prove that $EM(X)$ is convex in $AM(X)$, take any $d, \rho \in EM(X)$ and any $t \in [0,1]$.
By the convexity of $AM(X)$, $(1 - t)d + t\rho \in AM(X)$.
Then for every $n \in \N$,
\begin{align*}
 \diam_{(1 - t)d + t\rho}(X \setminus X_n) &= \sup_{(x,y) \in (X \setminus X_n)^2} ((1 - t)d(x,y) + t\rho(x,y))\\
 &\leq (1 - t)\sup_{(x,y) \in (X \setminus X_n)^2} d(x,y) + t\sup_{(x,y) \in (X \setminus X_n)^2} \rho(x,y)\\
 &= (1 - t)\diam_d(X \setminus X_n) + t\diam_\rho(X \setminus X_n).
\end{align*}
Due to Lemma~\ref{diam.}, $\diam_d(X \setminus X_n) \to 0$ and $\diam_\rho(X \setminus X_n) \to 0$ as $n \to \infty$,
 and hence $\diam_{(1 - t)d + t\rho}(X \setminus X_n)$ is also converging to $0$.
Using Lemma~\ref{diam.} again, we get that $(1 - t)d + t\rho \in EM(X)$.
Therefore $EM(X)$ is convex in $AM(X)$,
 and hence so is in $C(X^2)$.
When $X$ is locally connected,
 $C(X^2)$ is a Fr\'{e}chet space.
Then the latter part holds.
\end{proof}

We can show the following proposition.

\begin{prop}\label{homot.dense}
Let $X$ be a locally connected but not compact space.
Then $EM(X)$ is homotopy dense in $PM(X)$.
\end{prop}

\begin{proof}
The convex subset $PM(X) \subset C(X^2)$ contains $EM(X)$ as a dense convex subset by Propositions~\ref{dense} and \ref{cvx.}.
Combining Theorem~6.8.9 with Corollary~6.8.5 of \cite{Sakaik11}, we have that $EM(X)$ is homotopy dense in $PM(X)$.
\end{proof}

\section{The $Z_\sigma$-set property of $EM(X)$ in $PM(X)$}

This section is devoted to proving that $EM(X)$ is contained in some $Z_\sigma$-set in $PM(X)$.
For a compact metric space $Y = (Y,d_Y)$, we shall consider the set of partial pseudometrics on compact sets in $Y$.
Let
 $$PPM(Y) = \bigcup \{PM(A) \mid A \in \cpt(Y)\}$$
 whose topology is defined as follows:
Identifying each partial pseudometrics $d \in PPM(Y)$ with its graph
 $$\{(x,y,d(x,y)) \mid x, y \in \dom{d}\} \in \cpt(Y \times Y \times \R),$$
 where the symbol $\dom{d} \in \cpt(Y)$ stands for the domain of $d$,
 we can regard $PPM(Y)$ as a subspace of $\cpt(Y \times Y \times \R)$.
Note that the Hausdorff metric $(d_{Y \times Y \times \R})_H$ is admissible on $PPM(Y)$.
Here set
 $$PAM(Y) = \bigcup \{AM(A) \mid A \in \cpt(Y) \text{ and } A \text{ is non-degenerate}\} \subset PPM(Y).$$

\begin{lem}\label{conti.}
Let $Y = (Y,d_Y)$ be a compact metric space, $U$ be an open subset of $Y$, $K$ be a closed subset of $Y$,
 and $y_0 \in U$ and $y_\infty \in K$ be distinct points.
Suppose that $Z$ is a space,
 and $f : Z \to \cpt(Y)$, $g : Z \to \cpt(Y)$ and $h : Z \to [0,\infty)$ are maps such that $y_0 \in f(z) \subset U$, $y_\infty \in g(z) \subset K$ and $f(z) \cap g(z) = \emptyset$ for every $z \in Z$.
Moreover, let $d_0 : Z \to PM(U)$ and $d_\infty : Z \to PM(K)$ be maps.
Then the function $\Phi : Z \to PPM(Y)$ is continuous,
 which is defined by
\begin{enumerate}
 \item $\dom{\Phi(z)} = f(z) \cup g(z)$;
 \item $\Phi(z)(x,y) = \left\{
 \begin{array}{ll}
 d_0(z)(x,y) &\text{if } x, y \in f(z),\\
 d_\infty(z)(x,y) &\text{if } x, y \in g(z),\\
 d_0(z)(x,y_0) + d_\infty(z)(y,y_\infty) + h(z) &\text{if } x \in f(z) \text{ and } y \in g(z),\\
 d_0(z)(y,y_0) + d_\infty(z)(x,y_\infty) + h(z) &\text{if } x \in g(z) \text{ and } y \in f(z).
 \end{array}
 \right.$
\end{enumerate}
\end{lem}

\begin{proof}
As is easily observed,
 $\Phi(z) \in PPM(Y)$ for any $z \in Z$.
To verify the continuity of $\Phi$, take any $z \in Z$ and $\epsilon > 0$.
Since $f(z)$ is compact,
 there is $\delta \in (0,\epsilon)$ such that $\overline{B}_{d_Y}(f(z),\delta) \subset U$.
Combinig the continuity of $d_0(z)$ and $d_\infty(z)$ with the compactness of $f(z)$ and $g(z)$, we can assume that for each $(x,y) \in f(z)^2$ and each $(x',y') \in \overline{B}_{d_Y}(x,\delta) \times \overline{B}_{d_Y}(y,\delta)$, $|d_0(z)(x,y) - d_0(z)(x',y')| < \epsilon/6$,
 and for each $(x,y) \in g(z)^2$ and each $(x',y') \in (\overline{B}_{d_Y}(x,\delta) \cap K) \times (\overline{B}_{d_Y}(y,\delta) \cap K)$, $|d_\infty(z)(x,y) - d_\infty(z)(x',y')| < \epsilon/6$.
Since $f$, $g$ and $h$ are continuous,
 we can choose a neighborhood $V \subset Z$ of $z$ so that if $w \in V$,
 then $(d_Y)_H(f(z),f(w)) \leq \delta$, $(d_Y)_H(g(z),g(w)) \leq \delta$ and $|h(z) - h(w)| < \epsilon/3$.
Moreover, take a neighborhood $W \subset V$ of $z$ such that for any $w \in W$, if $(x,y) \in \overline{B}_{d_Y}(f(z),\delta)^2$,
 then $|d_0(z)(x,y) - d_0(w)(x,y)| < \epsilon/6$,
 and if $(x,y) \in K^2$,
 then $|d_\infty(z)(x,y) - d_\infty(w)(x,y)| < \epsilon/6$.
Observe that for each $w \in W$, $(d_{Y \times Y \times \R})_H(\Phi(z),\Phi(w)) < \epsilon$.
Indeed, let any $(x,y,\Phi(z)(x,y)) \in \Phi(z)$.
When $(x,y) \in f(z)^2$,
 there exists a point $(x',y') \in f(w)^2$ such that
 $$d_{Y^2}((x,y),(x',y')) = \max\{d_Y(x,x'),d_Y(y,y')\} \leq \delta < \epsilon$$
 because $(d_Y)_H(f(z),f(w)) \leq \delta$.
Note that $|d_0(z)(x,y) - d_0(z)(x',y')| < \epsilon/6$.
Moreover, $(x',y') \in \overline{B}_{d_Y}(f(z),\delta)^2$,
 and hence $|d_0(z)(x',y') - d_0(w)(x',y')| < \epsilon/6$.
Observe that
\begin{multline*}
 |\Phi(z)(x,y) - \Phi(w)(x',y')| = |d_0(z)(x,y) - d_0(w)(x',y')|\\
 \leq |d_0(z)(x,y) - d_0(z)(x',y')| + |d_0(z)(x',y') - d_0(w)(x',y')| < \epsilon/6 + \epsilon/6 < \epsilon.
\end{multline*}
Thus we have that $(x',y',\Phi(w)(x',y')) \in \Phi(w)$ and
\begin{multline*}
 d_{Y \times Y \times \R}((x,y,\Phi(z)(x,y)),(x',y',\Phi(w)(x',y')))\\
 = \max\{d_{Y^2}((x,y),(x',y')),|\Phi(z)(x,y) - \Phi(w)(x',y')|\} < \epsilon.
\end{multline*}
Similarly, when $(x,y) \in g(z)^2$,
 there exists a point $(x',y') \in g(w)^2$ such that $(x',y',\Phi(w)(x',y')) \in \Phi(w)$ and
 $$d_{Y \times Y \times \R}((x,y,\Phi(z)(x,y)),(x',y',\Phi(w)(x',y'))) < \epsilon.$$
When $x \in f(z)$ and $y \in g(z)$,
 we can find points $x' \in f(w)$ and $y' \in g(w)$ so that $d_Y(x,x') \leq \delta < \epsilon$ and $d_Y(y,y') \leq \delta < \epsilon$ since $(d_Y)_H(f(z),f(w)) \leq \delta$ and $(d_Y)_H(g(z),g(w)) \leq \delta$.
Recall that $|d_0(z)(x,y_0) - d_0(z)(x',y_0)| < \epsilon/6$ and $|d_\infty(z)(y,y_\infty) - d_\infty(z)(y',y_\infty)| < \epsilon/6$.
Furthermore, $|d_0(z)(x',y_0) - d_0(w)(x',y_0)| < \epsilon/6$ and $|d_\infty(z)(y',y_\infty) - d_\infty(w)(y',y_\infty)| < \epsilon/6$ because $(x',y_0) \in \overline{B}_Y(f(z),\delta)^2$ and $(y',y_\infty) \in K^2$.
Then
\begin{align*}
 &|\Phi(z)(x,y) - \Phi(w)(x',y')|\\
 &\ \ \ \ = |(d_0(z)(x,y_0) + d_\infty(z)(y,y_\infty) + h(z)) - (d_0(w)(x',y_0) + d_\infty(w)(y',y_\infty) + h(w))|\\
 &\ \ \ \ \leq |d_0(z)(x,y_0) - d_0(z)(x',y_0)| + |d_0(z)(x',y_0) - d_0(w)(x',y_0)|\\
 &\ \ \ \ \ \ \ \ + |d_\infty(z)(y,y_\infty) - d_\infty(z)(y',y_\infty)| + |d_\infty(z)(y',y_\infty) - d_\infty(w)(y',y_\infty)| + |h(z) - h(w)|\\
 &\ \ \ \ < \epsilon/6 + \epsilon/6 + \epsilon/6 + \epsilon/6 + \epsilon/3 = \epsilon.
\end{align*}
Therefore we have
\begin{multline*}
 d_{Y \times Y \times \R}((x,y,\Phi(z)(x,y)),(x',y',\Phi(w)(x',y')))\\
 = \max\{d_{Y^2}((x,y),(x',y')),|\Phi(z)(x,y) - \Phi(w)(x',y')|\} < \epsilon.
\end{multline*}
Similarly, when $x \in g(z)$ and $y \in f(z)$,
 there are points $x' \in g(w)$ and $y' \in f(w)$ such that $(x',y',\Phi(w)(x',y')) \in \Phi(w)$ and
 $$d_{Y \times Y \times \R}((x,y,\Phi(z)(x,y)),(x',y',\Phi(w)(x',y'))) < \epsilon.$$
It follows that $\Phi(z) \subset B_{d_{Y \times Y \times \R}}(\Phi(w),\epsilon)$.
By the same argument as the above, we can see that $\Phi(w) \subset B_{d_{Y \times Y \times \R}}(\Phi(z),\epsilon)$.
Consequently, $(d_{Y \times Y \times \R})_H(\Phi(z),\Phi(w)) < \epsilon$,
 which implies that $\Phi$ is continuous.
\end{proof}

\begin{remark}
In the above lemma, for each $z \in Z$, if $d_0(z)$ and $d_\infty(z)$ are admissible on $f(z)$ and $g(z)$ respectively, and $h(z) > 0$,
 then $\Phi(z) \in PAM(Y)$.
\end{remark}

We will give a useful path on $\cpt(X)$ for the latter argument.

\begin{lem}
Let $X$ be connected and locally connected.
Then there exists a map $\xi : [0,\infty) \to \cpt(X)$ satisfying the following conditions:
\begin{enumerate}
 \item $X = \bigcup \xi([0,\infty))$;
 \item For any $0 \leq s \leq t < \infty$, $\xi(s) \subset \xi(t)$;
 \item For each $n \in \N$, $\xi(n) \subset \intr{\xi(n + 1)}$.
\end{enumerate}
\end{lem}

\begin{proof}
We will prove this lemma in the case where $X \neq \emptyset$.
Write $X = \bigcup_{n \in \N} X_n$,
 where each $X_n$ is compact and $X_n \subset \intr{X_{n + 1}}$.
We may assume that $X_1 \neq \emptyset$,
 so choose any point $x_0 \in X_1$.
According to Proposition~\ref{hyp.}, the hyperspace $\cpt(X)$ is connected, locally connected and locally compact metrizable,
 and hence it is path-connected by Theorem~5.14.5 of \cite{Sakaik11}.
Now we shall inductively construct a map $\xi : [0,\infty) \to \cpt(X)$ so that $X_n \subset \xi(n)$ for each $n \in \N$.
Firstly, since $\cpt(X)$ is path-connected,
 there is a path $\xi_1 : [0,1] \to \cpt(X)$ such that $\xi_1(0) = \{x_0\}$ and $\xi_1(1) = X_1$.
Then we can define a map $\xi_1' : [0,1] \to \cpt(\cpt(X))$ by $\xi_1'(t) = \xi_1([0,t])$ for all $t \in [0,1]$.
Moreover, by virtue of \cite[Lemma~1.11.6 and Proposition~1.11.7]{Mil3}, we can obtain a map $\xi_1'' : [0,1] \to \cpt(X)$ defined by $\xi_1''(t) = \bigcup \xi'(t)$ for each $t \in [0,1]$.
As is easily observed,
 for any $0 \leq s \leq t \leq 1$, $\xi_1''(s) \subset \xi_1''(t)$ and $X_1 \subset \xi_1''(1)$.
Then let $\xi|_{[0,1]} = \xi_1''$.
Secondly, assume that $\xi|_{[0,n]}$ is obtained for some $n \in \N$.
Because $\xi(n)$ is compact,
 there exists $m \geq n + 1$ such that $\xi(n) \subset \intr{X_m}$.
Due to the same argument as the above, extend $\xi|_{[0,n]}$ over $[0,n + 1]$ so that $X_m \subset \xi(n + 1)$.
By the inductive construction, we can get the desired map $\xi$.
\end{proof}

From now on, the map $\xi$ is as the above lemma and let $X_n = \xi(n)$ for each $n \in \N$.
Moreover, if $X$ is not empty,
 then fix a point $x_0 \in X_1$.
We shall define a distance $D$ between real-valuded functions on $X^2$ as follows:
 $$D(f,g) = \sup_{n \in \N} \min\bigg\{\sup_{(x,y) \in X_n^2}|f(x,y) - g(x,y)|, 1/n\bigg\}$$
 for any $f, g : X^2 \to \R$.
Note that $D$ is an admissible metric on $C(X^2)$, refer to \cite[1.1.3(7)]{Sakaik11}.
We can see the following:

\begin{lem}\label{approx.}
Suppose that $X$ is connected and locally connected,
 and that $\epsilon > 0$.
For any $f, g \in C(X^2)$, if $f|_{\xi(1/\epsilon)^2} = g|_{\xi(1/\epsilon)^2}$,
 then $D(f,g) \leq \epsilon$.
\end{lem}

\begin{proof}
Let $f, g \in C(X^2)$ with $f|_{\xi(1/\epsilon)^2} = g|_{\xi(1/\epsilon)^2}$.
In the case that $n \leq 1/\epsilon$,
 $X_n = \xi(n) \subset \xi(1/\epsilon)$,
 so
 \begin{align*}
 \min\bigg\{\sup_{(x,y) \in X_n^2}|f(x,y) - g(x,y)|, 1/n\bigg\} &\leq \sup_{(x,y) \in X_n^2}|f(x,y) - g(x,y)|\\
 &= \sup_{(x,y) \in X_n^2}|f(x,y) - f(x,y)| = 0.
 \end{align*}
In the case that $n > 1/\epsilon$,
 $$\min\bigg\{\sup_{(x,y) \in X_n^2}|f(x,y) - g(x,y)|, 1/n\bigg\} \leq 1/n < \epsilon.$$
Therefore we have
 $$D(f,g) = \sup_{n \in \N} \min\bigg\{\sup_{(x,y) \in X_n^2}|f(x,y) - g(x,y)|, 1/n\bigg\} \leq \epsilon.$$
The proof is complete.
\end{proof}

We shall use the following extension theorem of partial metrics with various domains according to \cite[Theorem~2.1]{BSTZ}.

\begin{thm}\label{ext.}
Let $Y$ be compact.
There exists a map $e : PPM(Y) \to PM(Y)$ such that $PAM(Y) \subset AM(Y)$.
\end{thm}

For spaces $Y \subset M$, the restriction $r : PM(M) \to PM(Y)$ is continuous,
 which is defined by $r(d) = d|_{Y^2}$ for all $d \in PM(M)$.
Note that $r(AM(M)) \subset AM(Y)$.
From now on, let $e : e : PPM(\alpha X) \to PM(\alpha X)$ be the extension as in Theorem~\ref{ext.} and let $r : PM(\alpha X) \to PM(X)$ be the restriction as the above.

\begin{prop}\label{Zsigma}
Let $X$ be connected and locally connected, but not compact.
The space $AM(X)$ is contained in some $Z_\sigma$-set in $PM(X)$,
 and hence so is $EM(X)$.
\end{prop}

\begin{proof}
Notice that $AM(X) \subset \bigcup_{m \in \N} A(1,m)$.
To show that $A(1,m)$ is a $Z$-set in $PM(X)$ for every $m \in \N$, fix any map $\epsilon : PM(X) \to (0,1)$.
We will construct a map $\Psi : PM(X) \to PM(X)$ so that $\Psi(PM(X)) \cap A(1,m) = \emptyset$ and $D(d,\Psi(d)) \leq \epsilon(d)$ for each $d \in PM(X)$.
Define a function $\Phi : PM(X) \to PPM(\alpha X)$ satisfying the following conditions:
\begin{enumerate}
 \item $\dom{\Phi(d)} = \xi(1/\epsilon(d)) \cup \{x_\infty\}$;
 \item $\Phi(d)(x,y) = \left\{
 \begin{array}{ll}
 d(x,y) &\text{if } x, y \in \xi(1/\epsilon(d)),\\
 d(x,x_0) + 1/(m + 1) &\text{if } x \in \xi(1/\epsilon(d)) \text{ and } y = x_\infty,\\
 d(y,x_0) + 1/(m + 1) &\text{if } x = x_\infty \text{ and } y \in \xi(1/\epsilon(d)),\\
 0 &\text{if } x = y = x_\infty.
 \end{array}
 \right.$
\end{enumerate}
It follows from Lemma~\ref{conti.} that $\Phi$ is continuous.
Let $\Psi = re\Phi$,
 that is a desired map.
To verify it, take any $d \in PM(X)$.
Firstly, we shall show that $\Psi(d) \notin A(1,m)$.
Since $e\Phi(d)$ is a continuous pseudometric on $\alpha X$ and $x_\infty$ is not an isolated point,
 we can obtain $x \in X \setminus X_2$ so that $e\Phi(d)(x_\infty,x) < 1/m - 1/(m + 1)$.
Then
\begin{align*}
 \inf\{\Psi(d)(y,z) \mid y \in X_1, z \in X \setminus X_2\} &\leq \Psi(d)(x_0,x) = re\Phi(d)(x_0,x) = e\Phi(d)(x_0,x)\\
 &\leq e\Phi(d)(x_0,x_\infty) + e\Phi(d)(x_\infty,x)\\
 &< 1/(m + 1) + 1/m - 1/(m + 1) = 1/m,
\end{align*}
 which means that $\Psi(d) \notin A(1,m)$.
Secondly, prove that $D(d,\Psi(d)) \leq \epsilon(d)$.
Remark that
 $$\Psi(d)(x,y) = re\Phi(d)(x,y) = e\Phi(d)(x,y) = \Phi(d)(x,y) = d(x,y)$$
 for any $x, y \in \xi(1/\epsilon(d))$.
It follows from Lemma~\ref{approx.} that $D(d,\Psi(d)) \leq \epsilon(d)$.
We conclude that $A(1,m)$ is a $Z$-set in $PM(X)$,
 so $AM(X)$ is contained in the $Z_\sigma$-set $\bigcup_{m \in \N} A(1,m)$.
\end{proof}

\section{The strong $\mathfrak{M}_2$-universality of $EM(X)$}

In this section, we will verify the strong $\mathfrak{M}_2$-universality of $EM(X)$.
For any $c > 0$, let
 $$Z(c) = \bigg\{d \in PM(X) \bigg| \liminf_{n \to \infty} d(x_0,X \setminus X_n) \leq c\bigg\}.$$
We have the following lemma.

\begin{lem}\label{raise}
Let $X$ be connected and locally connected, but not compact.
Suppose that $A$ is a closed set and $B$ is a $Z$-set in $PM(X)$, respectively.
If $A \subset Z(c) \cup B$ for some $c > 0$,
 then $A$ is a $Z$-set.
\end{lem}

\begin{proof}
For each map $\epsilon : PM(X) \to (0,1)$, we shall construct a map $\Psi : PM(X) \to PM(X)$ such that $\Psi(PM(X)) \cap (Z(c) \cup B) = \emptyset$ and $D(d,\Psi(d)) \leq \epsilon(d)$ for any $d \in PM(X)$.
Since $B$ is a $Z$-set,
 there exists a map $\Phi_1 : PM(X) \to PM(X)$ such that $\Phi_1(PM(X)) \cap B = \emptyset$ and $D(d,\Phi_1(d)) \leq \epsilon(d)/2$.
Letting
 $$\delta(d) = \min\{\epsilon(d),D(\Phi_1(d),B)\}/2,$$
 we define a function $\Phi_2 : PM(X) \to PPM(\alpha X)$ as follows:
\begin{enumerate}
 \item $\dom{\Phi_2(d)} = \xi(1/\delta(d)) \cup \{x_\infty\}$;
 \item $\Phi_2(d)(x,y) = \left\{
 \begin{array}{ll}
 d(x,y) &\text{if } x, y \in \xi(1/\delta(d)),\\
 d(x,x_0) + c + 1 &\text{if } x \in \xi(1/\delta(d)) \text{ and } y = x_\infty,\\
 d(y,x_0) + c + 1 &\text{if } x = x_\infty \text{ and } y \in \xi(1/\delta(d)),\\
 0 &\text{if } x = y = x_\infty.
 \end{array}
 \right.$
\end{enumerate}
Due to Lemma~\ref{conti.}, $\Phi_2$ is continuous.
Then we can obtain the desired map $\Psi = re\Phi_2\Phi_1$.
By the same argument as Proposition~\ref{Zsigma},
 $$D(\Phi_1(d),\Psi(d)) \leq \delta(d) = \min\{\epsilon(d),D(\Phi_1(d),B)\}/2$$
 for each $d \in PM(X)$.
Therefore
 $$D(d,\Psi(d)) \leq D(d,\Phi_1(d)) + D(\Phi_1(d),\Psi(d)) \leq \epsilon(d)/2 + \delta(d) \leq \epsilon(d),$$
 and $\Psi(PM(X)) \cap B = \emptyset$.
It remains to prove that $\Psi(PM(X)) \cap Z(c) = \emptyset$.
For each $d \in PM(X)$, since $e\Phi_2\Phi_1(d) \in PM(\alpha X)$,
 there exists $m \in \N$ such that if $x \in X \setminus X_m$,
 then $e\Phi_2\Phi_1(d)(x_\infty,x) < 1/2$.
It follows that
\begin{align*}
 \Psi(d)(x_0,x) &= re\Phi_2\Phi_1(d)(x_0,x) = e\Phi_2\Phi_1(d)(x_0,x)\\
 &\geq e\Phi_2\Phi_1(d)(x_0,x_\infty) - e\Phi_2\Phi_1(d)(x_\infty,x) > c + 1 - 1/2 = c + 1/2
\end{align*}
 for any $k \geq m$ and any $x \in X \setminus X_k$.
Hence $\liminf_{n \to \infty} \Psi(d)(x_0,X \setminus X_n) > c$,
 which means that $\Psi(PM(X)) \cap Z(c) = \emptyset$.
As a result, $A$ is a $Z$-set.
\end{proof}

We show the following:

\begin{lem}\label{chain}
Let $X$ be a locally connected but not compact space.
Then $X$ is chain-connected to infinity if and only if there is an arc $\sigma : [0,1] \to \alpha X$ such that $\sigma(0) = x_\infty$ and $\sigma((0,1]) \subset X$.
\end{lem}

\begin{proof}
Firstly, we shall prove the ``if'' part.
Taking an arc $\sigma : [0,1] \to \alpha X$ so that $\sigma(0) = x_\infty$ and $\sigma((0,1]) \subset X$, we can find $t_n \in (0,1]$ for each $n \in \N$ such that $\sigma((0,t_n]) \subset X \setminus X_n$ and $t_{n + 1} < t_n$.
Let $C_i = \sigma((0,t_i])$,
 so $\{C_i\}$ is a simple chain consisting of connected subset in $X$.
Moreover, for every compact set $K \subset X$, there is $i(K) \in \N$ such that $K \subset X_{i(K)}$,
 and hence if $i \geq i(K)$,
 then $C_i \subset X \setminus X_{i(K)} \subset X \setminus K$.
It follows that $X$ is chain-connected to infinity.

Next, we show the ``only if'' part.
Since $X$ is chain-connected to infinity,
 we can obtain a simple chain $\{C_i\}$ consisting of connected sets in $X$ so that for every $n \in \N$, there is $i(n) \in \N$ such that for any $i \geq i(n)$, $C_i \subset X \setminus X_n$.
Put $D_n = \bigcup_{i \geq i(n)} C_i$ and observe that for any $n \in \N$, $D_n$ is connected and $D_{n + 1} \subset D_n \subset X \setminus X_n$.
Then we may assume that each $D_n$ is an open subset of $X$.
Indeed, replace $D_1$ with the connected component in $X \setminus X_1$ containing $D_1$.
Then the connected component $D_1$ is open because $X$ is locally connected.
Suppose that $D_n$ is open for some $n \in \N$.
Replacing $D_{n + 1}$ with the connected component in $D_n \cap X \setminus X_n$ containing $D_{n + 1}$, we have that $D_{n + 1}$ is open due to the local connectedness of $X$.
By induction, every $D_n$ is open.
Fix any point $x_n \in D_n$ for each $n \in \N$,
 so we can choose a path $\sigma|_{[1/(n + 1),1/n]} : [1/(n + 1),1/n] \to D_n$ so that $\sigma(1/n) = x_n$ and $\sigma(1/(n + 1)) = x_{n + 1}$ since $D_n$ is path-connected.
Connect them,
 so we get the map $\sigma|_{(0,1]} : (0,1] \to X$,
 that can be extended to the path $\sigma : [0,1] \to \alpha X$ by $\sigma(0) = x_\infty$.
Replacing $\sigma$ with an arc, we finish the proof.
\end{proof}

The next lemma will be used for proving Proposition~\ref{str.univ.}.

\begin{lem}\label{arc}
Suppose that $X$ is connected and locally connected, but not compact.
If $X$ is chain-connected to infinity,
 then there exists an arc $\gamma : [0,1] \to \alpha X$ such that the following conditions hold:
\begin{enumerate}
 \item $\gamma(0) = x_\infty$ and $\gamma((0,1]) \subset X$,
 \item $\xi(t) \cap \gamma([0,1/t]) = \emptyset$ for all $t \in [1,\infty)$.
\end{enumerate}
\end{lem}

\begin{proof}
By virtue of Lemma~\ref{chain}, we can choose an arc $\sigma : [0,1] \to \alpha X$ so that $\sigma(0) = x_\infty$ and $\sigma([0,1]) \subset \alpha X \setminus X_2$.
For each $n \in \N$, there is $t_n \in (0,1]$ such that $\sigma([0,t_n]) \subset \alpha X \setminus X_{n + 1}$ and $t_n > t_{n + 1}$.
Indeed, let $t_1 = 1$ and assume that $t_n \in (0,1]$ can be obtained for some $n \in \N$.
Due to the continuity at $0$ of $\sigma$, we can find $t_{n + 1} \in (0,t_n)$ such that $\sigma([0,t_{n + 1}]) \subset \alpha X \setminus X_{n + 2}$.
By induction, we have the decreasing sequence $\{t_n\} \subset (0,1]$.
Let $\tau : [0,1] \to [0,1]$ be the map such that for each $n \in \N$ and each $t \in [1/(n + 1),1/n]$,
 $$\tau(t) = (t - 1/(n + 1))n(n + 1)(t_n - t_{n + 1}) + t_{n + 1}.$$
Then the desired arc $\gamma$ can be defined by $\gamma(t) = \sigma\tau(t)$.
\end{proof}

Using the arc $\gamma : [0,1] \to \alpha X$ as in the above lemma, we define an arc $\eta : [0,1] \to \cpt(\alpha X)$ by $\eta(t) = \gamma([0,t])$.
Let $x_m = \gamma(2^{-m}) \in X$ for each $m \in \N$.
A closed set $A \subset Y$ is called to be a \textit{strong $Z$-set} if the identity map of $Y$ can be approximated by a map $f : Y \to Y$ such that $\cl{f(Y)} \cap A = \emptyset$.
Combining Theorem~\ref{pm}, Proposition~\ref{homot.dense} with \cite[Theorem~1.3.2 and Proposition~1.4.3]{BRZ}, we have the following:

\begin{prop}\label{str.Z}
Suppose that $X$ is a locally connected but not compact space.
If $X$ is not discrete,
 then every $Z$-set in $EM(X)$ is a strong $Z$-set.
\end{prop}

The space
 $$\cs_1 = \bigg\{(x(n))_{n \in \N} \in \Q \ \bigg| \ \lim_{n \to \infty} x(n) = 1\bigg\}$$
 is an $\mathfrak{M}_2$-absorbing set in $\Q$,
 see to \cite{Mil3},
 and hence it admits closed embeddings from spaces belonging to $\mathfrak{M}_2$.
Now we show the following:

\begin{prop}\label{str.univ.}
Let $X$ be connected and locally connected, but not compact.
Suppose that $X$ is chain-connected to infinity.
Then the space $EM(X)$ is strongly $\mathfrak{M}_2$-universal.
\end{prop}

\begin{proof}
Suppose that $A \in \mathfrak{M}_2$, $B$ is a closed set in $A$, and $f : A \to EM(X)$ is a map such that $f|_B$ is a $Z$-embedding.
For each open cover $\mathcal{U}$ of $EM(X)$, let us construct a $Z$-embedding $h : A \to EM(X)$ such that $h$ is $\mathcal{U}$-close to $f$ and $h|_B = f|_B$.
By Proposition~\ref{cvx.}, $EM(X)$ is an AR.
Remark that $X$ is not discrete.
According to Proposition~\ref{str.Z}, $f(B)$ is a strong $Z$-set in $EM(X)$.
By virtue of \cite[Proposition~2.8.12 and Lemma~2.8.10]{Sakaik12}, we can assume that $f(A \setminus B) \cap f(B) = \emptyset$ and $f$ satisfies the following property:
\begin{enumerate}
 \renewcommand{\labelenumi}{(\roman{enumi})}
 \item For each metric $d \in f(B)$ and each sequence $\{a_n\} \subset A$, if $f(a_n)$ is converging to $d$,
 then $a_n$ is converging to $f^{-1}(d)$.
\end{enumerate}
Take a map $\epsilon : EM(X) \to [0,1)$ such that
\begin{enumerate}
 \renewcommand{\labelenumi}{(\roman{enumi})}
 \setcounter{enumi}{1}
 \item for any map $h : A \to EM(X)$, if $D(f(a),h(a)) \leq \epsilon f(a)$ for every $a \in A$,
 then $h$ is $\mathcal{U}$-close to $f$,
 \item for each $d \in EM(X)$, $\epsilon(d) \leq D(d,f(B))/2$,
 and $\epsilon(d) = 0$ if and only if $d \in f(B)$.
\end{enumerate}

For each $k \in \N$, we put
 $$A_k = \{a \in A \mid 2^{-k} \leq \epsilon f(a) \leq 2^{-k + 1}\}$$
 and $\phi_k(a) = 2 - 2^k\epsilon f(a)$.
Then $A \setminus B = \bigcup_{k \in \N} A_k$.
Fixing a closed embedding $p : A \to \cs_1$, we can define a map $q_i^k : A_k \to [0,1]$ by
 $$q_i^k(a) = \left\{
 \begin{array}{ll}
  0 &\text{if } i = 1,\\
  \epsilon f(a)(1 - \phi_k(a)) &\text{if } i = 2,\\
  \epsilon f(a)(1 - \phi_k(a))p(a)(1) & \text{if } i = 3,\\
  \epsilon f(a) &\text{if } i = 2j, j \geq 2,\\
  \epsilon f(a)((1 - \phi_k(a))p(a)((i - 1)/2) + \phi_k(a)p(a)((i - 3)/2)) &\text{if } i = 2j + 1, j \geq 2.
 \end{array}
 \right.$$
For each $m \in \N$, let $\psi_m : \gamma([2^{-m},2^{-m + 1}]) \to [0,1]$ be a map defined by $\psi_m(x) = 2^m\gamma^{-1}(x) - 1$.
Define a map $g_k : A_k \to PM(\{x_0\} \cup \gamma([0,1]))$, $k \in \N$, as follows:
\begin{multline*}
 g_k(a)(x,x_0) = g_k(a)(x_0,x)\\
 = \left\{
 \begin{array}{ll}
  \epsilon f(a) &\text{if } x = x_\infty,\\
  \psi_{2k + i}(x)q_i^k(a) + (1 - \psi_{2k + i}(x))q_{i + 1}^k(a) &\text{if } x \in \gamma([2^{-2k - i},2^{-2k - i + 1}]),\\
  0 &\text{if } x = x_0 \text{ or } 2^{-2k} \leq \gamma^{-1}(x) \leq 1,\\
 \end{array}
 \right.
\end{multline*}
 and for any $x, y \in \{x_0\} \cup \gamma([0,1])$, $g_k(a)(x,y) = |g_k(a)(x,x_0) - g_k(a)(y,x_0)|$.
Here we will observe that $g_k(a) = g_{k + 1}(a)$ for all $a \in A_k \cap A_{k + 1}$.
Indeed, $g_k(a)(x_\infty,x_0) = \epsilon f(a) = g_{k + 1}(a)(x_\infty,x_0)$,
 and $g_k(a)(x,x_0) = 0 = g_{k + 1}(a)(x,x_0)$ for every $x \in \{x_0\} \cup \gamma([0,1])$ with $x = x_0$ or $2^{-2k} \leq \gamma^{-1}(x) \leq 1$.
Because $\epsilon f(a) = 2^{-k}$,
 $q_1^k(a) = q_2^k(a) = q_3^k(a) = 0$.
Thus for each $x \in \gamma([2^{-2k - 1},2^{-2k}])$,
 $$g_k(a)(x,x_0) = \psi_{2k + 1}(x)q_1^k(a)+(1 - \psi_{2k + 1}(x))q_2^k(a) = 0 = g_{k + 1}(a)(x,x_0),$$
 and for each $x \in \gamma([2^{-2k - 2},2^{-2k - 1}])$,
 $$g_k(a)(x,x_0) = \psi_{2k + 2}(x)q_2^k(a)+(1 - \psi_{2k + 2}(x))q_3^k(a) = 0 = g_{k + 1}(a)(x,x_0).$$
Furthermore, $q_3^k(a) = 0 = q_1^{k + 1}(a)$, $q_{2j + 3}^k(a) = \epsilon f(a)p(a)(j) = q_{2j + 1}^{k + 1}(a)$ and $q_{2j + 2}^k(a) = \epsilon f(a) = q_{2j}^{k + 1}(a)$ for every $j \geq 1$,
 and hence for all $x \in \gamma([2^{-2k - i - 2},2^{-2k - i - 1}])$, $i \geq 1$,
\begin{align*}
 g_k(a)(x,x_0) &= \psi_{2k + i + 2}(x)q_{i + 2}^k(a) + (1 - \psi_{2k + i + 2}(x))q_{i + 3}^k(a)\\
 &= \psi_{2(k + 1) + i}(x)q_i^{k + 1}(a) + (1 - \psi_{2(k + 1) + i}(x))q_{i + 1}^{k + 1}(a) = g_{k + 1}(a)(x,x_0).
\end{align*}
It follows that $g_k(a) = g_{k + 1}(a)$.
Letting $g : A \setminus B \to AM(\{x_0\} \cup \gamma([0,1]))$ be a map defined by
 $$g(a)(x,y) = \left\{
 \begin{array}{ll}
 g_k(a)(x,y) + |\gamma^{-1}(x) - \gamma^{-1}(y)| &\text{if } x, y \in \gamma([0,1]),\\
 g_k(a)(y,x_0) + \gamma^{-1}(y) + 1 &\text{if } x = x_0 \text{ and } y \in \gamma([0,1]),\\
 g_k(a)(x,x_0) + \gamma^{-1}(x) + 1 &\text{if } x \in \gamma([0,1]) \text{ and } y = x_0,\\
 g_k(a)(x,y) &\text{if } x = y = x_0,
 \end{array}
 \right.$$
 where $a \in A_k$,
 we can obtain a function $g' : A \setminus B \to PAM(\alpha X)$ so that for any $a \in A \setminus B$,
\begin{enumerate}
 \item $\dom{g'(a)} = \xi(1/\epsilon f(a)) \cup \eta(\epsilon f(a))$;
 \item $g'(a)(x,y) = \left\{
 \begin{array}{ll}
 f(a)(x,y) &\text{if } x, y \in \xi(1/\epsilon f(a)),\\
 g(a)(x,y) &\text{if } x, y \in \{x_0\} \cup \eta(\epsilon f(a)),\\
 f(a)(x,x_0) + g(a)(y,x_0) &\text{if } x \in \xi(1/\epsilon f(a)) \text{ and } y \in \eta(\epsilon f(a)),\\
 f(a)(y,x_0) + g(a)(x,x_0) &\text{if } x \in \eta(\epsilon f(a)) \text{ and } y \in \xi(1/\epsilon f(a)).
 \end{array}
 \right.$
\end{enumerate}
The continuity of $g'$ follows from Lemma~\ref{conti.}.

Let $h|_{A \setminus B} : A \setminus B \to EM(X)$ be a map defined by $h(a) = reg'(a)$.
Note that $D(f(a),h(a)) \leq \epsilon f(a)$ by Lemma~\ref{approx.}.
According to (iii), the map $h$ can be extended to the desired map $h : A \to EM(X)$ by $h|_B = f|_B$,
 and verify that
 $$h(A \setminus B) \subset EM(X) \setminus f(B) = EM(X) \setminus h(B).$$
By (ii), $h$ is $\mathcal{U}$-close to $f$.
It is remains to show that $h$ is a $Z$-embedding.
To check the injectivity of $h$, fix any $a_1, a_2 \in A \setminus B$ with $h(a_1) = h(a_2)$,
 where we get some $k_1, k_2 \in \N$ such that $a_1 \in A_{k_1}$ and $a_2 \in A_{k_2}$ respectively.
Let $k = \max\{k_l \mid i = 1, 2\}$.
Remark that for each $i \in \N$ and for each $l = 1, 2$,
\begin{align*}
 q_{2(k - k_l) + i + 1}^{k_l}(a_l) &= \psi_{2k + i}(x_{2k + i})q_{2(k - k_l) + i}^{k_l}(a_l) + (1 - \psi_{2k + i}(x_{2k + i}))q_{2(k - k_l) + i + 1}^{k_l}(a_l)\\
 &= g_{k_l}(a_l)(x_{2k + i},x_0) = g(a_l)(x_{2k + i},x_0) - (2^{-2k - i} + 1)\\
 &= h(a_l)(x_{2k + i},x_0) - (2^{-2k - i} + 1),
\end{align*}
 so when $i = 3$,
 $$\epsilon f(a_1) = q_{2(k - k_1) + 4}^{k_1}(a_1) = q_{2(k - k_2) + 4}^{k_2}(a_2) = \epsilon f(a_2).$$
Hence $a_1, a_2 \in A_k$ and $\phi_k(a_1) = \phi_k(a_2)$.
In the case where $\phi_k(a_1) = 1$,
 $$p(a_1)(j) = q_{2j + 3}^k(a_1)/\epsilon f(a_1) = q_{2j + 3}^k(a_2)/\epsilon f(a_2) = p(a_2)(j)$$
 for any $j \in \N$.
In the case where $\phi_k(a_1) \neq 1$,
 $$p(a_1)(1) = q_3^k(a_1)/((1 - \phi_k(a_1))\epsilon f(a_1)) = q_3^k(a_2)/((1 - \phi_k(a_2))\epsilon f(a_2)) = p(a_2)(1).$$
Assuming that $p(a_1)(j) = p(a_2)(j)$ for some $j \in \N$, we see that
\begin{align*}
 p(a_1)(j + 1) &= (q_{2j + 3}^k(a_1)/\epsilon f(a_1) - \phi_k(a_1)p(a_1)(j))/(1 - \phi_k(a_1))\\
 &= (q_{2j + 3}^k(a_2)/\epsilon f(a_2) - \phi_k(a_2)p(a_2)(j))/(1 - \phi_k(a_2)) = p(a_2)(j + 1).
\end{align*}
By induction, $p(a_1)(j) = p(a_2)(j)$ for all $j \in \N$.
Consequently, $p(a_1) = p(a_2)$.
The injectivity of $h$ follows from the one of $p$.

We prove that $h$ is a closed map.
For any sequence $\{a_n\} \subset A$ and any metric $d \in EM(X)$ such that $h(a_n) \to d$ as $n \to \infty$, we will choose a subsequence of $\{a_n\}$ converging to some point in $A$.
Notice that
 $$D(f(a_n),d) \leq D(f(a_n),h(a_n)) + D(h(a_n),d) \leq \epsilon f(a_n) + D(h(a_n),d).$$
When $\epsilon f(a_n) \to 0$,
 $D(f(a_n),d) \to 0$.
Due to the continuity of $\epsilon$, $\epsilon(d) = 0$,
 so $d \in f(B)$ by (iii).
It follows from (i) that $a_n$ converges to $f^{-1}(d)$.
When $\epsilon f(a_n) \nrightarrow 0$,
 we may replace $\{a_n\}$ with a subsequence in $A_k = (\epsilon f)^{-1}([2^{-k},2^{-k + 1}])$ for some $k \in \N$.
By the compactness of $[2^{-k},2^{-k + 1}]$,
 we can also replace $\{a_n\}$ with a subsequence such that $\epsilon f(a_n)$ converges to some number $c \in [2^{-k},2^{-k + 1}]$.
For each $j \in \N$,
\begin{align*}
 c &= \lim_{n \to \infty} \epsilon f(a_n) = \lim_{n \to \infty} (h(a_n)(x_{2(k + j) + 1},x_0) - (2^{-2(k + j) - 1} + 1))\\
 &= d(x_{2(k + j) + 1},x_0) - (2^{-2(k + j) - 1} + 1).
\end{align*}
The metric $d \in EM(X)$ can be extended to $\overline{d} \in AM(\alpha X)$.
Therefore
 $$\overline{d}(x_\infty,x_0) = \lim_{j \to \infty} d(x_{2(k + j) + 1},x_0) = \lim_{j \to \infty} (c + (2^{-2(k + j) - 1} + 1)) = c + 1.$$
In the case that $c = 2^{-k}$, $\lim_{n \to \infty} \phi_k(a_n) = 1$,
 so we can assume that $\phi_k(a_n) > 0$.
Then
 \begin{align*}
 &\lim_{n \to \infty} p(a_n)(j)\\
 &\ \ = \lim_{n \to \infty} (h(a_n)(x_{2(k + j + 1)},x_0) - (2^{-2(k + j + 1)} + 1) - \epsilon f(a_n)(1 - \phi_k(a_n))p(a_n)(j + 1))/(\epsilon f(a_n)\phi_k(a_n))\\
 &\ \ = 2^k(d(x_{2(k + j + 1)},x_0) - (2^{-2(k + j + 1)} + 1))
 \end{align*}
 for every $j \in \N$.
Note that as $j \to \infty$,
 $$2^k(d(x_{2(k + j + 1)},x_0) - (2^{-2(k + j + 1)} + 1)) \to 2^k(\overline{d}(x_\infty,x_0) - 1) = 2^k((2^{-k} + 1) - 1) = 1,$$
 which implies that $p(a_n)$ converges to $\{2^k(d(x_{2(k + j + 1)},x_0) - (2^{-2(k + j + 1)} + 1))\} \in \cs_1$.
Since $p$ is a closed embedding,
 $a_n$ is converging to some point of $A$.
By the same argument, $a_n$ is converging to some point of $A$ in the case that $c = 2^{-k + 1}$.
When $c \in (2^{-k},2^{-k + 1})$,
 it can be assumed that $\phi_k(a_n) \in (0,1)$,
 and hence as $n \to \infty$,
\begin{align*}
 p(a_n)(1) &= (h(a_n)(x_{2(k + 1)},x_0) - (2^{-2(k + 1)} + 1))/(\epsilon f(a_n)(1 - \phi_k(a_n)))\\
 &\to (d(x_{2(k + 1)},x_0) - (2^{-2(k + 1)} + 1))/(c(2^kc - 1)).
\end{align*}
Suppose that for some $j \in \N$, $p(a_n)(j)$ is converging and let $p_j = \lim_{n \to \infty} p(a_n)(j)$.
Then
\begin{align*}
 &\lim_{n \to \infty} p(a_n)(j + 1)\\
 &\ \ \ \ = \lim_{n \to \infty} (h(a_n)(x_{2(k + j + 1)},x_0) - (2^{-2(k + j + 1)} + 1) - \epsilon f(a_n)\phi_k(a_n)p(a_n)(j))/(\epsilon f(a_n)(1 - \phi_k(a_n)))\\
 &\ \ \ \ = (d(x_{2(k + j + 1)},x_0) - (2^{-2(k + j + 1)} + 1) - c(2 - 2^kc)p_j)/(c(2^kc - 1)).
\end{align*}
By induction, for all $j \in \N$, $p(a_n)(j)$ converges and denote $p_j = \lim_{n \to \infty} p(a_n)(j)$.
Recall that $\{p_j\} \in \Q$.
We will prove that $\{p_j\} \in \cs_1$.
Supposing that $p_j \nrightarrow 1$, we can choose a positive number $\lambda \leq 2 - 2^kc$ so that for any $j \in \N$, there is $i(j) \geq j$ such that $p_{i(j)} < 1 - \lambda$.
For every $j \in \N$,
\begin{align*}
 (2^kc - 1)p_{j + 1} + (2 - 2^kc)p_j &= \lim_{n \to \infty} ((1 - \phi_k(a_n))p(a_n)(j + 1) + \phi_k(a_n)p(a_n)(j))\\
 &= \lim_{n \to \infty} q_{2j + 3}^k(a_n)/\epsilon f(a_n)\\
 &= \lim_{n \to \infty} (h(a_n)(x_{2(k + j + 1)},x_0) - (2^{-2(k + j + 1)} + 1))/c\\
 &= (d(x_{2(k + j + 1)},x_0) - (2^{-2(k + j + 1)} + 1))/c.
\end{align*}
It follows that
\begin{align*}
 \lim_{j \to \infty} ((2^kc - 1)p_{j + 1} + (2 - 2^kc)p_j) &= \lim_{j \to \infty} (d(x_{2(k + j + 1)},x_0) - (2^{-2(k + j + 1)} + 1))/c\\
 &= (\overline{d}(x_\infty,x_0) - 1)/c = ((c + 1) - 1)/c = 1.
\end{align*}
Thus there exist $j \in \N$ such that if $i \geq j$,
 then $1 - \lambda^2 \leq (2^kc - 1)p_{i + 1} + (2 - 2^kc)p_i$.
Hence
 $$1 - \lambda^2 \leq (2^kc - 1)p_{i(j) + 1} + (2 - 2^kc)p_{i(j)} \leq (2 - 2^kc)p_{i(j)} + 2^kc - 1.$$
Since $\lambda \leq 2 - 2^kc$,
 we get that
 $$1 - \lambda \leq 1 - \lambda^2/(2 - 2^kc) \leq p_{i(j)},$$
 which is a contradiction.
Thus $\{p_j\} \in \cs_1$.
Then $p(a_n)$ converges to $\{p_j\} \in \cs_1$ as $n \to \infty$,
 so $a_n$ is converging to some point of $A$ because $p$ is a closed embedding.
It follows that $h$ is a closed map.

For every $a \in A \setminus B$, $a \in A_k$ for some $k \in \N$,
 and hence for any $i \in \N$, since
 \begin{align*}
 h(a)(x_{2(k + i) + 1},x_0) &= g(a)(x_{2(k + i) + 1},x_0) = g_k(a)(x_{2(k + i) + 1},x_0) + (2^{-2(k + i) - 1} + 1)\\
 &= q_{2(i + 1)}^k(a) + (2^{-2(k + i) - 1} + 1) = \epsilon f(a) + (2^{-2(k + i) - 1} + 1) < 3,
 \end{align*}
 we have $\liminf_{n \to \infty} h(a)(x_0,X \setminus X_n) \leq 3$.
According to Lemma~\ref{raise}, the image $h(A) = h(A \setminus B) \cup h(B)$, that is contained in $Z(3) \cup f(B)$,
 is a $Z$-set in $EM(X)$.
We conclude that $h$ is a $Z$-embedding.
\end{proof}

\section{Proof of Main Theorem}

Now we shall prove Main Theorem.

\begin{proof}[Proof of Main Theorem]
By Proposition~\ref{M2}, we have $EM(X) \in \mathfrak{M}_2 \subset (\mathfrak{M}_2)_\sigma$.
Due to Propositions~\ref{homot.dense}, $EM(X)$ is homotopy dense in $PM(X)$.
We can decompose $X = \bigoplus_{k \in K} Y_k$ into connected components for some $K \subset \N$.
Since $X$ is chain-connected to infinity,
 we can choose a positive integer $k \in K$ and a path between $x_\infty$ and some point of $Y_k$ by Lemma~\ref{chain}.
Hence $Y_k$ is chain-connected to infinity and $Y_k \cup \{x_\infty\}$ is the one-point compactification of $Y_k$.
Remark that $C(X^2)$ is homeomorphic to the product space $\prod_{(i,j) \in K^2} C(Y_i \times Y_j)$ by virtue of the following homeomorphism:
 $$C(X^2) \ni f \mapsto (f|_{Y_i \times Y_j})_{(i,j) \in K^2} \in \prod_{(i,j) \in K^2} C(Y_i \times Y_j).$$
Let $PM_{(i,j)} = \{d|_{Y_i \times Y_j} \mid d \in PM(X)\}$ and $EM_{(i,j)} = \{d|_{Y_i \times Y_j} \mid d \in EM(X)\}$,
 so $PM(X)$ is homeomorphic to $PM_{(k,k)} \times \prod_{(i,j) \in K^2 \setminus \{(k,k)\}} PM_{(i,j)}$ and $EM(X)$ is homeomorphic to $EM_{(k,k)} \times \prod_{(i,j) \in K^2 \setminus \{(k,k)\}} EM_{(i,j)}$, respectively.
Note that for each $(i,j) \in K^2$, $PM_{(i,j)}$ and $EM_{(i,j)}$ are convex sets in a Fr\'{e}chet space $C(Y_i \times Y_j)$,
 so they are ARs,
 and moreover the product spaces $\prod_{(i,j) \in K^2 \setminus \{(k,k)\}} PM_{(i,j)}$ and $\prod_{(i,j) \in K^2 \setminus \{(k,k)\}} EM_{(i,j)}$ are also ARs.
Observe that $PM(Y_k) = PM_{(k,k)}$ and $EM(Y_k) = EM_{(k,k)}$.
Here we only verify the latter equality.
Clearly, $EM(Y_k) \supset EM_{(k,k)}$.
To show that $EM(Y_k) \subset EM_{(k,k)}$, let any $d \in EM(Y_k)$.
We use the same symbol for the extension of $d$ on $Y_k \cup \{x_\infty\}$.
Fix an admissible metric $\rho \in AM(\alpha X)$,
 so we can define an extension $\overline{d} \in AM(\alpha X)$ of $d$ as follows:
 $$\overline{d}(x,y) = \left\{
 \begin{array}{ll}
 d(x,y) &\text{if } x, y \in Y_k \cup \{x_\infty\},\\
 \rho(x,y) &\text{if } x \in \alpha X \setminus Y_k,\\
 d(x,x_\infty) + \rho(y,x_\infty) &\text{if } x \in Y_k \text{ and } y \in X \setminus Y_k,\\
 d(y,x_\infty) + \rho(x,x_\infty) &\text{if } x \in X \setminus Y_k \text{ and } y \in Y_k.
 \end{array}
 \right.$$
Thus $d \in EM_{(k,k)}$.
By Proposition~\ref{Zsigma}, there exists a $Z_\sigma$-set $Z$ in $PM(Y_k)$ that contains $EM(Y_k)$.
It follows from \cite[Proposition~3.5]{Kos15} that $Z \times \prod_{(i,j) \in K^2 \setminus \{(k,k)\}} PM_{(i,j)}$ is a $Z_\sigma$-set in $PM_{(k,k)} \times \prod_{(i,j) \in K^2 \setminus \{(k,k)\}} PM_{(i,j)}$,
 which means that $EM(X)$ is contained in some $Z_\sigma$-set of $PM(X)$.
According to Proposition~\ref{str.univ.}, $EM(Y_k)$ is strongly $\mathfrak{M}_2$-universal.
Combining \cite[Proposition~2.6]{BeMo} with Proposition~\ref{str.Z}, we can see that the product space $EM_{(k,k)} \times \prod_{(i,j) \in K^2 \setminus \{(k,k)\}} EM_{(i,j)}$ is strongly $\mathfrak{M}_2$-universal,
 and therefore so is $EM(X)$.
Hence the space $EM(X)$ is an $\mathfrak{M}_2$-absorbing set in $PM(X)$.
Combining this with Theorems~\ref{pm} and \ref{abs.}, we conclude that $EM(X)$ is homeomorphic to $\cs_0$.
\end{proof}

\end{document}